\documentclass[10pt,a4paper]{article}         

\usepackage{color}
\usepackage{a4,amsmath,amssymb,amscd,latexsym,amsthm} 
\usepackage{bbm} 
\usepackage{epsfig,color} 
\usepackage{amsfonts}
\usepackage{mathrsfs}
\usepackage{overpic}
\usepackage{subcaption}
\usepackage{paralist}
\usepackage{graphicx}
\usepackage{trfsigns}
\usepackage{url}
\usepackage{stmaryrd}

\usepackage{tikz}
\usetikzlibrary{calc,trees,positioning,arrows,chains,shapes.geometric,%
    decorations.pathreplacing,decorations.pathmorphing,shapes,%
    matrix,shapes.symbols,backgrounds,shadows}
  
\usepackage{tikz-cd} 

\newtheoremstyle{mystyle}
{3pt}
{3pt}
{\itshape}
{}
{\bold}
{.}
{.5em}
{}

\newtheorem{definition}{Definition}[section]
\newtheorem{theorem}[definition]{Theorem}
\newtheorem{remark}[definition]{Remark}

\numberwithin{equation}{section}

\newcommand{\Gi}{\mathrm{\Gamma_{int}}}
\newcommand{\Gb}{\mathrm{\Gamma_{bottom}}}
\newcommand{\Gl}{\mathrm{\Gamma_{left}}}
\newcommand{\Gr}{\mathrm{\Gamma_{right}}}
\newcommand{\Gs}{\mathrm{\Gamma_{sides}}}
\newcommand{\Gt}{\mathrm{\Gamma_{top}}}
\newcommand{\Gfront}{\mathrm{\Gamma_{front}}}
\newcommand{\Gback}{\mathrm{\Gamma_{back}}}
\newcommand{\Go}{\mathrm{\Gamma_{out}}}

\def\nd#1{\frac{\partial #1}{\partial n}} 
\def\td#1{\frac{\partial #1}{\partial t}} 

\def\intdxdt#1{\int_{0}^{T} \int_{\Omega} #1 \, dx\, dt}

\def\intdx#1{\int_{\Omega} #1 \, dx}
\def\intdsidt#1{\int_{0}^{T}\int_{\Gi} #1 \, ds\, dt}

\def\intdsodt#1{\int_{0}^{T} \int_{\Go} #1 \, ds\, dt}

\def\intdt#1{\int_{0}^{T} #1 \, dt}

\newcommand{\LL}{{\mathscr{L}}}

\newcommand{\R}{{\mathbb{R}}} 

\makeatletter
\def\moverlay{\mathpalette\mov@rlay}
\def\mov@rlay#1#2{\leavevmode\vtop{%
   \baselineskip\z@skip \lineskiplimit-\maxdimen
   \ialign{\hfil$\m@th#1##$\hfil\cr#2\crcr}}}
\newcommand{\charfusion}[3][\mathord]{
    #1{\ifx#1\mathop\vphantom{#2}\fi
        \mathpalette\mov@rlay{#2\cr#3}
      }
    \ifx#1\mathop\expandafter\displaylimits\fi}
\makeatother

\newcommand{\cupdot}{\charfusion[\mathbin]{\cup}{\cdot}}
\newcommand{\bigcupdot}{\charfusion[\mathop]{\bigcup}{\cdot}}

\begin{document}
\title{\bf{Algorithmic aspects of multigrid methods for optimization in shape spaces}}
\author{Martin Siebenborn\thanks{Universit\"at Trier, D-54286 Trier, Germany, Email: siebenborn@uni-trier.de, welker@uni-trier.de} \and Kathrin Welker\footnotemark[1]
}
\date{}
\maketitle

\begin{abstract}
\noindent
We examine the interaction of multigrid methods and shape optimization in appropriate shape spaces.
Our aim is a scalable algorithm for application on supercomputers, which can only be achieved by mesh-independent convergence.
The impact of discrete approximations of geometrical quantities, like the mean curvature, on a multigrid shape optimization algorithm with quasi-Newton updates is investigated.
For the purpose of illustration, we consider a complex model for the identification of cellular structures in biology with minimal compliance in terms of elasticity and diffusion equations.
\end{abstract}

\section{Introduction}
PDE constrained shape optimization is becoming more and more suited for practical applications, e.g., acoustics \cite{Berggren2016largescale}, aerodynamics \cite{Mohammadi-2001} and electrostatics \cite{Langer-2015}. 
A finite dimensional optimization problem can be obtained for example by representing shapes as splines.
However, the connection of shape calculus with infinite dimensional spaces \cite{Delfour-Zolesio-2001,ItoKunisch,SokoZol} leads to a more flexible approach. 
In recent work, it has been shown that PDE constrained shape optimization problems can be embedded in the framework of optimization on shape spaces.
Finding a shape space and an associated metric is a challenging task and different approaches lead to various models. There exists no common shape space suitable for all applications. 
One possible approach is to define shapes as elements of a Riemannian manifold as proposed in \cite{MichorMumford2,MichorMumford1}.
In \cite{MichorMumford}, a survey of various suitable inner products is given, e.g., the curvature weighted Riemannian metric and the Sobolev metric.
From a theoretical point of view this is attractive because algorithmic ideas from \cite{Absil} can be combined with approaches from differential geometry. 
For example in \cite{schulz2014structure}, shape optimization is considered as optimization on a Riemannian shape manifold.
This particular manifold contains only shapes with infinitely differentiable boundaries, which limits the practical applicability.
From a computational point of view, usually finite element methods are used to discretize the PDE models where one has to deal with polygonal shape representations. 

A well-established approach is to deal with shape derivatives in a so-called Hadamard form, i.e., in the form of integrals over the surface (cf.~\cite{Mohammadi-2001,SokoZol}).
An equivalent and intermediate result in the process of deriving Hadamard expressions is a volume expression of the shape derivative.
One usually has to require additional regularity assumptions in order to transform volume into surface forms.
In addition to saving analytical effort this makes volume expressions preferable over Hadamard forms, which is also utilized in \cite{Langer-2015}.
In the case of the more attractive volume formulation, the shape manifold and the corresponding inner products mentioned above are not appropriate.
One possible approach to use these formulations is given in \cite{schulz2015Steklov}.
Here an inner product, which is called Steklov-Poincar\'{e} metric, and a shape space are proposed.
These are further considered in the following.
The combination of this particular shape space and its associated inner product is an essential step towards applying efficient finite element solvers.
This is especially important in order to avoid complicated load balancing for parallel computers with respect to both, the volume and surface elements. 
If there are many computationally costly operations on surfaces like the evaluation of the Sobolev metric, this is necessary.
By using volume forms and Steklov-Poincar\'{e}-type metrics one only has to consider standard finite element assemblies and a classical load balancing with respect to volume elements.

In general, practical applications necessitate very fine discretizations.
Here one can observe mesh-dependence at several points, e.g., the number of iterations of the PDE solver and the overall optimization loop.
The same holds for discrete approximations of geometrical quantities, like the mean curvature (cf.~\cite{meyer2003discrete}), which strongly depend on the chosen mesh.
Especially a perimeter regularization is affected by the increasing values of the discretized mean curvature within a multgrid framework.
In particular, scalable algorithms for application on supercomputers can only be achieved by mesh-independent convergence of both, the PDE simulation and the overall optimization loop.
In order to come up with such an algorithm, one has to apply on the one hand multigrid preconditioner to the PDE solver and on the other hand quasi-Newton methods or something even more sophisticated for the optimization.
In this paper, we examine the interaction of multigrid and shape optimization based on Steklov-Poincar\'{e} metrics and the corresponding shape space.
We focus on this particular combination because it is well-suited for a large-scale finite element algorithm with quasi-Newton techniques.

This paper has the following structure. 
In section \ref{section_model}, besides a short review on the background of shape optimization, PDE models are formulated.
Furthermore, the corresponding shape derivatives are given.
Section \ref{section_algorithm} summarizes the entire process from shape derivatives to a complete optimization algorithm in suitable shape spaces.
A numerical test framework and results are presented in section \ref{section_results}.

\section{Model formulations and their shape derivatives}\label{section_model}
After setting up notation and terminology we formulate the model problems. They are motivated by finding the optimal design of cellular structures. In the third part of this section we deduce the shape derivatives of the model problems.

\subsection{Notations and definitions}

One focus in shape optimization is to investigate shape functionals. First, we define such a functional.

\begin{definition}[Shape functional]
Let $D$ denote a non-empty subset of $\R^d$, where $d\in\mathbb{N}$. Moreover, $\mathcal{A}\subset \{\Omega\colon \Omega \subset D\}$ denotes a set of subsets. A function
$$J\colon \mathcal{A}\to \R\text{, } \Omega\mapsto J(\Omega)$$
is called a shape functional.
\end{definition}

\noindent
In the following, let $D$ be as in the definition above. 
Moreover, let $\{F_t\}_{t\in[0,T]}$ be a family of mappings $F_t\colon \overline{D}\to\mathbb{R}^d$ such that $F_0=id$, where $T>0$.
This family transforms a domain $\Omega\subset D$ into \emph{perturbed domains} 
\begin{equation*}
\Omega_t := F_t(\Omega)=\{F_t(x)\colon x\in \Omega\}\text{ with }\Omega_0=\Omega
\end{equation*}
and the boundary $\Gamma=\partial\Omega$ into \emph{perturbed boundaries}
\begin{equation*}
\Gamma_t:= F_t(\Gamma)=\{F_t(x)\colon x\in \Gamma\}\text{ with }\Gamma_0=\Gamma.
\end{equation*}
Considering the domain $\Omega$ as a collection of material particles, which are changing their position in the time-interval $[0,T]$, the family $\{F_t\}_{t\in[0,T]}$ describes the motion of each particle. This means that at time $t\in [0,T]$ a particle $x\in\Omega$ has the new position $x_t:= F_t(x)\in\Omega_t$ with $x_0=x$.
The motion of each such particle $x$ can be described by the \emph{velocity method} or by the \emph{perturbation of identity}.
Let $V$ be a sufficiently smooth vector field.
For $V$ the velocity method defines the family of the above-mentioned mappings as the flow $F_t(x):= \xi(t,x)$, which is determined by the following initial value problem:
\begin{equation*}
\begin{split}
\frac{d\xi(t,x)}{dt}&=V(\xi(t,x))\\
\xi(0,x)&=x
\end{split}
\end{equation*}
The perturbation of identity is defined by $F_t(x):= x+tV(x)$ and used in the following.

This paper deals with PDE constrained shape optimization problems, i.e., shape optimization problems constrained by equations involving an unknown function of two or more variables and at least one partial derivative of this function. Such a problem is given by
\begin{equation*}
\min_{\Omega} J(\Omega),
\end{equation*}
where $J$ is a shape functional additionally depends on a solution of a PDE.
To solve these problems, we need their shape derivatives: 

\begin{definition}[Shape derivative]
\label{def_shapeder}
Let $D\subset \R^d$ be open, where $d\geq 2$ is a natural number. Moreover, let $k\in\mathbb{N}\cup \{\infty\}$ and let $\Omega\subset D$ be measurable.
The Eulerian derivative of a shape functional $J$ at $\Omega$ in direction $V\in\mathcal{C}^k_0(D,\R^d)$ is defined by
\begin{equation}
\label{eulerian}
DJ(\Omega)[V]:= \lim\limits_{t\to 0^+}\frac{J(\Omega_t)-J(\Omega)}{t}. 
\end{equation}
If for all directions $V\in\mathcal{C}^k_0(D,\R^d)$ the Eulerian derivative (\ref{eulerian}) exists and the mapping 
\begin{equation*}
G(\Omega)\colon \mathcal{C}^k_0(D,\R^d)\to \R, \ V\mapsto DJ(\Omega)[V]
\end{equation*}
is linear and continuous, the expression $DJ(\Omega)[V]$ is called the shape derivative of $J$ at $\Omega$ in direction $V\in\mathcal{C}^k_0(D,\R^d)$. In this case, $J$ is called shape differentiable of class $\mathcal{C}^k$ at $\Omega$.
\end{definition}

For a detailed introduction into shape calculus, we refer to the monographs \cite{Delfour-Zolesio-2001,SokoZol}.
In particular, \cite{SokoZol} states that shape derivatives can always be expressed as boundary integrals due to the Hadamard structure theorem \cite[theorem 2.27]{SokoZol}.
In many cases, the shape derivative arises in two equivalent notational forms:
\begin{align}
\label{dom_form}
DJ_\Omega[V]&:=\int_\Omega F(x)V(x)\, dx  &\text{(volume formulation)}\\
\label{bound_form}
DJ_\Gamma[V]&:=\int_\Gamma f(s)V(s)^T n(s)\, ds  &\text{(surface formulation)}
\end{align}
Here $F$ is a (differential) operator acting linearly on the vector field $V$ and $f\in L^1(\Gamma)$ with $DJ_\Omega[V]=DJ(\Omega)[V]=DJ_\Gamma[V]$.

In general, we have to deal with so-called material derivatives in order to derive shape derivatives. The definition of these shape derivatives is given in the following. For a material derivative free approach we refer to \cite{Sturm2013}.

\begin{definition}[Material derivative]
\label{material_deriv}
Let $\Omega,\Omega_t,F_t$ and $T$ be as above. Moreover, let $\{p_t\colon \Omega_t\to \R\colon t\leq T \}$ denote a family of mappings. The material derivative of a generic function $p\hspace{.3mm}(=p_0)\colon \Omega\to \mathbb{R}$ at $x\in\Omega$ is denoted by $D_mp$ or $\dot{p}$ and given by the derivative of the composed function $p_t\circ F_t\colon \Omega\to\Omega_t\to\mathbb{R}$ defined in the fixed domain $\Omega$, i.e.,
\begin{equation*}
\dot{p}(x):=\lim\limits_{t\to 0^+}\frac{\left(p_t\circ F_t\right)(x)-p(x)}{t} =\frac{d^+}{dt}\left(p_t\circ F_t\right)(x)\,\rule[-2.5mm]{.1mm}{6mm}_{\hspace{1mm}t=0}.
\end{equation*}
\end{definition}

\noindent
In the following, let $p\colon \Omega\to\mathbb{R}$ be a function. 
The classical chain rule for differentiation applied to $\dot{p}$ gives the relation between material and shape derivatives. The shape derivative of $p$ in the direction of a vector field $V$ is denoted by $p'$ and given by
\begin{equation}
\label{shape_der}
p'=\dot{p}-V^T\nabla p \quad \text{in }\Omega.
\end{equation}
In subsection \ref{subsection_shapederivatives}, the following rules for the material derivative are needed to derive shape derivatives of objective functions depending on solutions of PDEs. For the material derivative the product rule holds, i.e.,
\begin{equation}
\label{product_rule}
D_m(p\hspace{.7mm}q)=D_mp\hspace{.7mm}q+p\hspace{.7mm}D_mq.
\end{equation}
While the shape derivative commutes with the gradient, the material derivative does not, but the following equality, which is proven in \cite{Berggren}, holds:
\begin{equation}
\label{material_grad}
D_m\nabla p=\nabla D_mp-\nabla V^T\nabla p.
\end{equation}

The concept of material and shape derivatives of a function $p\colon \Omega\to \R$ can be extended to its boundary $\Gamma=\partial\Omega$. We mention only a few aspects. For more details we refer to the literature, e.g., \cite{SokoNov}. Let $z\colon \Gamma\to \R$ be the trace on the boundary $\Gamma$ of $p$. In this setting, the boundary shape derivative $z'$ is defined by
\begin{equation}
\label{shape_der_boundary}
z'=\dot{p}-V^T\nabla_\Gamma p,
\end{equation}
where $\nabla_\Gamma$ denotes the tangential gradient given by
\begin{equation*}
\nabla_\Gamma p=\nabla p-\frac{\partial p}{\partial n}n.
\end{equation*}
Here $\frac{\partial}{\partial n}$ denotes the derivative normal to $\Gamma$.
Combining (\ref{shape_der}) with (\ref{shape_der_boundary}) gives the correlation of boundary and domain shape derivatives:
\begin{equation*}
z'=p'+V^T\frac{\partial p}{\partial n}n
\end{equation*}

In order to deduce shape derivative formulas we have to consider the derivative of perturbed objective functions:
\begin{align*}
\frac{d^+}{dt}\left(\int_{\Omega_t}p_t\ dx_t\right)\,\rule[-4mm]{.1mm}{9mm}_{\hspace{1mm}t=0} \qquad \text{and}\qquad \frac{d^+}{dt}\left(\int_{\Gamma_t}z_t\ ds_t\right)\,\rule[-4mm]{.1mm}{9mm}_{\hspace{1mm}t=0},
\end{align*}
where $\Omega,\Omega_t,p,p_t,\Gamma_t$ are as above and $z_t\colon\Gamma_t\to \R$ denotes a mapping. They are given by
\begin{align}
\label{der_domain_int}
\frac{d^+}{dt}\left(\int_{\Omega_t}p_t\ dx_t\right)\,\rule[-4mm]{.1mm}{9mm}_{\hspace{1mm}t=0} & =\int_{\Omega}\dot{p}+\mathrm{div}(V)p\ dx \text{,}\\
\label{der_boundary_int}
\frac{d^+}{dt}\left(\int_{\Gamma_t}z_t\ ds_t\right)\,\rule[-4mm]{.1mm}{9mm}_{\hspace{1mm}t=0} & =\int_{\Gamma}\dot{z}+\mathrm{div}_\Gamma(V)z\ ds,
\end{align}
where
\begin{equation*}
\mathrm{div}_\Gamma(V) := \text{div} (V) - n^T \frac{\partial V}{\partial n}
\end{equation*}
denotes the tangential divergence of $V$.
For their proofs we refer to the literature, e.g., \cite{HaslingerMakinen,ItoKunisch,Welker}.


\subsection{Problem formulations}
Let $\Omega\subset D$ be a bounded Lipschitz domain with boundary $\partial\Omega$. This domain is assumed to be a cuboid and partitioned in a subdomain $\Omega_\text{out}\subset\Omega$ and a finite number of disjoint subdomains $\Omega_i\subset \Omega$ with boundaries $\Gamma_i:=\partial\Omega_i$ such that $\Omega_\text{out}\bigcupdot\left(\cupdot_{i\in\mathbb{N}} \Omega_i\right) \bigcupdot\left(\cupdot_{i\in\mathbb{N}}\Gamma_i\right)=\Omega$ and $\Gb\cupdot\Gs\cupdot\Gt=\partial\Omega\ (=:\Go)$, where $\cupdot$ denotes the disjoint union.
The union of all domains $\Omega_i$ is denoted by $\Omega_\text{int} := \bigcupdot_{i\in\mathbb{N}}\Omega_i$ and the union of all boundaries $\Gamma_i$ is called the interface and denoted by $\Gi:=\bigcupdot_{i\in\mathbb{N}}\Gamma_i$. The outer normal vector field to $\Omega_\text{int}$ is given by $n$.  Figure \ref{fig_domain} illustrates this situation.
In our setting, $\Omega$ is meant to be composed of two distinct materials, one in $\Omega_\text{int}$ and one in $\Omega_\text{out}$.

Let $\nu_l\in\R$ be arbitrary constants, where $l\in\{1,\cdots,4\}$. For the objective function
\begin{equation}
J(y,u,\Omega) = j_1(u,\Omega)+j_2(y,\Omega)+j_3(\Omega)+j_4(\Omega)\label{objective}
\end{equation}
with
\begin{align*}
j_1(u,\Omega)&:=\nu_1 \int_{\Omega} \sigma(u):\epsilon(u)\; dx,\\
j_2(y,\Omega) &:=\frac{\nu _2}{2} \int_0^T \int_{\Omega} (y - \bar{y})^2\; dx\, dt,\\
j_3(\Omega)& := \nu_3 \int_{\Omega_\text{out}} 1\; dx ,\\
j_4(\Omega) & := \nu_4 \int_{\Gi} 1\; ds
\end{align*}
we consider the following PDE constrained optimization problem in strong form:
\begin{align}
\min\limits_{\Gi}\; J(y,u,\Omega)\label{eq_minimization}
\end{align}
subject to the following constraints:
\begin{align}
- \text{div}\,\sigma(u) & = 0 \quad \text{in}\; \Omega \label{le1} \\
u & = 0 \quad \text{on}\; \Gb \label{le2} \\
\sigma(u) \cdot n & = f \quad \text{on}\; \Gt \cup \Gs \label{le4}\\[10pt]
\frac{\partial y}{\partial t} - \text{div} (k \nabla y) & = 0 \quad \text{in} \; \Omega \times (0,T] \label{diff1}\\
y & = 1 \quad \text{on}\; \Gt \times (0,T] \label{diff2} \\
\frac{\partial y}{\partial n} & = 0  \quad \text{on}\; \left(\Gb\cup\Gs\right) \times (0,T] \label{diff3} \\
y & = 0  \quad \text{in} \; \Omega \times \lbrace 0 \rbrace \label{diff4}
\end{align}
Equations \eqref{le1}-\eqref{le4} describe the linear elasticity model, where $\sigma := \lambda \text{Tr}(\epsilon) I + 2\mu \epsilon$ denotes the stress tensor and $\epsilon = \frac{1}{2}\left( \nabla u + \nabla u^T \right)$ the strain tensor with respect to the Lam\'{e} parameters $\lambda$ and $\mu$.
We further assume
\begin{equation}
f = \begin{cases}
0 &\quad \text{on}\; \Gs\\
f_\text{top} &\quad \text{on}\; \Gt\\
\end{cases},
\end{equation}
where $f_\text{top}$ is a force at the top boundary of the domain, which leads to the deformation $u$.
A diffusion model is given by  \eqref{diff1}-\eqref{diff4} with a jumping permeability coefficient $k$.
The different properties of the two materials with respect to elasticity and permeability are modelled by the following choice of coefficients:
\begin{equation*}
k :=
\begin{cases}
k_\text{out} &\text{in} \; \Omega_\text{out}\\
k_\text{int} &\text{in} \; \Omega_\text{int}\\
\end{cases},\quad
\lambda :=
\begin{cases}
\lambda_\text{out} & \text{in} \; \Omega_\text{out}\\
\lambda_\text{int} & \text{in} \; \Omega_\text{int}\\
\end{cases},\quad
\mu :=
\begin{cases}
\mu_\text{out} & \text{in} \; \Omega_\text{out}\\
\mu_\text{int} & \text{in} \; \Omega_\text{int}\\
\end{cases}.
\end{equation*}
Due to the jump in these coefficients, the formulations (\ref{le1}) and (\ref{diff1}) are to be understood only formally.
The objective function $J$ is composed of the four shape functionals $j_l$. Here $j_1$ corresponds to the minimization of the compliance of the composite material in the domain $\Omega$. With $j_2$ the diffusion model is fitted to data measurements $\bar{y}$. We assume for the observation $\bar{y}\in L^2\left(0,T;L^2 (\Omega)\right)$.
Since we are interested in a space-filling design of the inclusions $\Omega_\text{int}$, the purpose of the functional $j_3$ is to minimize the volume of $\Omega_\text{out}$.
This is equivalent to maximizing $\Omega_\text{int}$ if the outer shape $\Go$ is fixed.
$j_4$ is a perimeter regularization. 
Of course, the integral in $j_4$ has to be understood as the sum of the integrals over $\Gi$. 
We formulate explicitly the continuity of the state and of the flux at the interface as
\begin{align}
\label{interface_cond1}
\left\llbracket u\right\rrbracket & =0, & \hspace{-2cm} \left\llbracket \sigma(u)\cdot n\right\rrbracket =0& \hspace{1cm} \text{on }\Gi,\\[5pt]
\label{interface_cond2}
\left\llbracket y\right\rrbracket & =0, & \hspace{-2cm} \left\llbracket k\nd{y}\right\rrbracket =0& \hspace{1cm}\text{on }\Gi\times(0,T],
\end{align}
where the jump symbol $\llbracket\cdot\rrbracket$ denotes the discontinuity across the interface $\Gi$ and is defined by $\llbracket v \rrbracket := v\,\rule[-2mm]{.1mm}{4mm}_{\hspace{.5mm}\Omega_\text{int}}-v\,\rule[-2mm]{.1mm}{4mm}_{\hspace{.5mm}\Omega_\text{out}}$ for $v\in\Omega$.

\begin{figure}
\begin{center}
\def\svgwidth{0.7\textwidth}
\begingroup%
  \makeatletter%
  \providecommand\color[2][]{%
    \errmessage{(Inkscape) Color is used for the text in Inkscape, but the package 'color.sty' is not loaded}%
    \renewcommand\color[2][]{}%
  }%
  \providecommand\transparent[1]{%
    \errmessage{(Inkscape) Transparency is used (non-zero) for the text in Inkscape, but the package 'transparent.sty' is not loaded}%
    \renewcommand\transparent[1]{}%
  }%
  \providecommand\rotatebox[2]{#2}%
  \ifx\svgwidth\undefined%
    \setlength{\unitlength}{324.88537598bp}%
    \ifx\svgscale\undefined%
      \relax%
    \else%
      \setlength{\unitlength}{\unitlength * \real{\svgscale}}%
    \fi%
  \else%
    \setlength{\unitlength}{\svgwidth}%
  \fi%
  \global\let\svgwidth\undefined%
  \global\let\svgscale\undefined%
  \makeatother%
  \begin{picture}(1,0.80890068)%
    \put(0,0){\includegraphics[width=\unitlength,page=1]{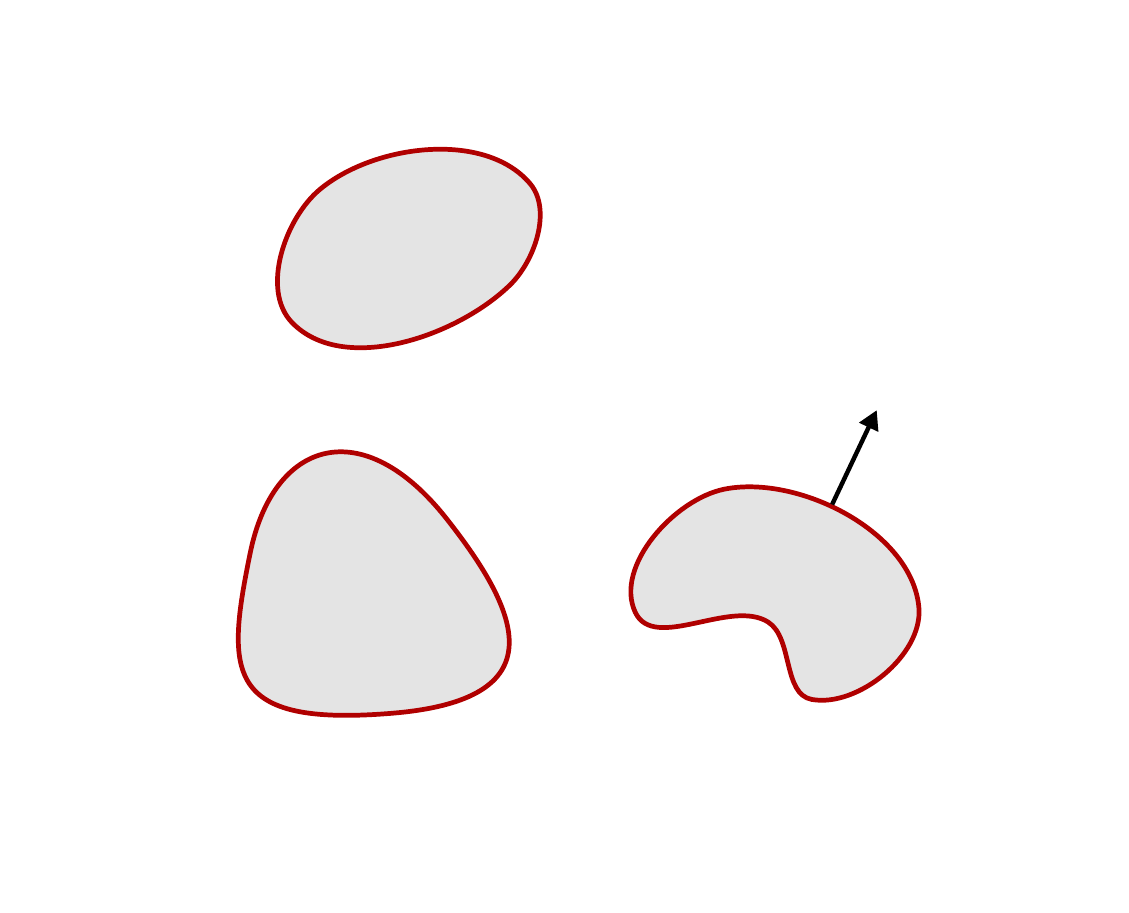}}%
    \put(0.33092989,0.58533828){\color[rgb]{0,0,0}\makebox(0,0)[lb]{\smash{$\Omega_1$}}}%
    \put(0.29751155,0.27226084){\color[rgb]{0,0,0}\makebox(0,0)[lb]{\smash{$\Omega_2$}}}%
    \put(0.63099177,0.5522717){\color[rgb]{0,0,0}\makebox(0,0)[lb]{\smash{$\Omega_\text{out}$}}}%
    \put(0.65913354,0.29371898){\color[rgb]{0,0,0}\makebox(0,0)[lb]{\smash{$\Omega_3$}}}%
    \put(0.78401269,0.41683932){\color[rgb]{0,0,0}\makebox(0,0)[lb]{\smash{$n$}}}%
    \put(0.42636477,0.00328166){\color[rgb]{0,0,0}\makebox(0,0)[lb]{\smash{$\Gamma_\text{bottom}$}}}%
    \put(0.45119061,0.7800747){\color[rgb]{0,0,0}\makebox(0,0)[lb]{\smash{$\Gamma_\text{top}$}}}%
    \put(0,0){\includegraphics[width=\unitlength,page=2]{domain.pdf}}%
    \put(0.37852368,0.37492926){\color[rgb]{0.69019608,0,0}\makebox(0,0)[lb]{\smash{$\Gamma_\text{int}$}}}%
  \end{picture}%
\endgroup%
\end{center}
\caption{Illustration of the domain $\Omega$}
\label{fig_domain}
\end{figure}

\begin{remark}
Equations (\ref{le1})-(\ref{le4}) admit a solution $u\in H^1(\Omega, \mathbb{R}^d)$ and (\ref{diff1})-(\ref{diff4}) have a solution $y\in L^2\left(0,T;H^1(\Omega)\right)$.
However, it turns out that the restrictions of these solutions to $\Omega_\text{int}$ and $\Omega_\text{out}$ have a higher regularity (cf.~for example \cite[proposition 11.7]{ItoKunisch}).
To be more precise, $u\,\rule[-2mm]{.1mm}{4mm}_{\hspace{.5mm}\Omega_\text{int}}\in H^2(\Omega_\text{int})$, $u\,\rule[-2mm]{.1mm}{4mm}_{\hspace{.5mm}\Omega_\text{out}}\in H^2(\Omega_\text{out}, \mathbb{R}^d)$, $y\,\rule[-2mm]{.1mm}{4mm}_{\hspace{.5mm}\Omega_\text{int}}\in L^2\left(0,T;H^2(\Omega_\text{int})\right)$ and $y\,\rule[-2mm]{.1mm}{4mm}_{\hspace{.5mm}\Omega_\text{out}}\in L^2\left(0,T;H^2(\Omega_\text{out})\right)$.
In this setting, the integral over $\Omega$ has to be understood as the sum of the integrals over $\Omega_\text{int}$ and $\Omega_\text{out}$.
In the following, the integral over $\Omega$ always denotes this sum.
Thus, it is guaranteed that $\frac{\partial u\,\rule[-.7mm]{.1mm}{2mm}_{\hspace{.5mm}\Omega_\text{int}}}{\partial n}\in H^{1/2}\left(\Gi, \mathbb{R}^d\right)$ and $\frac{\partial y\,\rule[-.7mm]{.1mm}{2mm}_{\hspace{.5mm}\Omega_\text{int}}}{\partial n}\in L^2\left(0,T;H^{1/2}(\Gi)\right)$ by the trace theorem for Sobolev spaces.
\end{remark}

The boundary value problem (\ref{le1})-(\ref{le4}) is given in its weak form by
\begin{equation}
a_1(u,w)=b_1(u,w_1, w_2)\, , \ \forall w\in H^1\left(\Omega, \mathbb{R}^d\right)
\end{equation}
and for all $w_1\in H^{-1/2}\left(\Gb, \mathbb{R}^d\right)$,
$w_2\in H^{1/2}\left(\Gs \cup\Gt, \mathbb{R}^d\right)$. Here the bilinear form $a_1(u,w)$ is given by
\begin{equation}
\begin{split}
a_1(u,w) & :=\int_{\Omega}\sigma(u):\epsilon(w)\hspace{.7mm}dx -\int_{\Gi}\left\llbracket \left(\sigma(u)\cdot n\right)^Tw\right\rrbracket dx \\
& \hspace*{.5cm} -\int_{\Go}\left(\sigma(u)\cdot n\right)^Tw \hspace{.7mm}ds
\end{split}
\end{equation}
and $b_1(u,w_1,w_2)$ is defined by
\begin{equation}
b_1(u,w_1,w_2) :=\int_{\Gb}w_1^T u\hspace{.7mm}ds +\int_{\Go\setminus\Gb}w_2^T \left(\sigma (u)\cdot n-f\right)\hspace{.3mm}ds.
\end{equation}
Moreover, problem (\ref{diff1})-(\ref{diff4}) is written in weak form as
\begin{equation}
\label{wf}
a_2(y,z)=b_2(y,z_1, z_2)\, , \ \forall z\in
W\left(0,T;H^1(\Omega)\right)
\end{equation}
and for all $z_1\in L^2\left(0,T;H^{-1/2}(\Gt)\right)$,
$z_2\in L^2\left(0,T;H^{1/2}(\Gb\cup\Gs)\right)$
as in \cite{schulz2014structure}. Here the bilinear form $a_2(y,z)$ is given by
\begin{equation}
\begin{split}
a_2(y,z) & :=\intdx{y(T,x)\, z(T,x)}-\intdxdt{\td{z}y+k\nabla y^T\nabla z}\\
& \hspace*{.5cm}-\intdsidt{\left\llbracket k\nd{y}z\right\rrbracket}-\intdsodt{k_\text{out}\nd{y}z}\label{wfa2}
\end{split}
\end{equation}
and $b_2(y,z_1,z_2)$ by
\begin{equation}
b_2(y,z_1,z_2) :=\intdt{\int_{\Gt}z_1(y-1)\, ds}+\intdt{\int_{\Go\setminus\Gt}z_2\nd{y}\, ds}.
\label{wfb}
\end{equation}
For properties of the function spaces, we refer to the literature, e.g., \cite{GrossReusken,Troeltzsch}.

The Lagrangian of (\ref{objective})-(\ref{diff4}) is defined by
\begin{equation}
\LL(y,u,w,z,\Omega)=\LL_1(u,w,\Omega)+\LL_2(y,z,\Omega)+\LL_3(\Omega)+\LL_4(\Omega)
\label{lagrangian}
\end{equation}
with 
\begin{align*}
\LL_1(u,w,\Omega)&:=j_1(u,\Omega) + a_1(u,w)-b_1(u,w_1,w_2) ,\\
\LL_2(y,z,\Omega) &:=j_2(y,\Omega) + a_2(y,z)-b_2(y,z_1,z_2),\\
\LL_3(\Omega)& := j_3(\Omega) ,\\
\LL_4(\Omega) & := j_4(\Omega).
\end{align*}

The adjoint problem to (\ref{le1})-(\ref{le4}), which we obtain from differentiating the Lagrangian $\LL_1$ with respect to $u$, is given in strong form by 
\begin{align}
-\text{div}\,\sigma(w) & = 0 \quad \text{in}\; \Omega \label{adjointle1} \\
w & = 0 \quad \text{on}\; \Gb \\
\sigma(w) \cdot n & = -\nu_1 f \quad \text{on}\; \Gs\cup\Gt \\
w_1&= \sigma(w) \cdot n \quad \text{on}\; \Gb \\
w_2 & = - w \quad \text{on}\; \Gs\cup\Gt \label{adjointle3}
\end{align}
and the state equation to (\ref{le1})-(\ref{le4}), which we get by differentiating the Lagrangian $\LL_1$ with respect to $w$, is given in strong form by
\begin{equation}
-\text{div}\,\sigma(u) = 0 \quad \text{in}\; \Omega. \label{designle}
\end{equation}
Moreover, the adjoint problem to (\ref{diff1})-(\ref{diff4}), which we obtain from differentiating the Lagrangian $\LL_2$ with respect to $y$, is given in strong form by 
\begin{align}
-\frac{\partial z}{\partial t}-\mathrm{div}(k\nabla z)&=-\nu_2(y-\overline{y}) \quad \text{in }\Omega \times [0,T)\label{adjointdiff1}\\
\nd{z}&=0\quad\text{on } \left(\Gb\cup\Gs\right) \times [0,T)\\
z&=0\quad\text{on }\Gt \times [0,T)\\
z&= 0\quad\text{in }\Omega \times \{T\}\\
z_1&=k_\text{out}\nd{z}\quad\text{on }\Gt \times [0,T)\\
z_2&=-k_\text{out}z\quad\text{on }\left(\Gb\cup\Gs\right) \times [0,T)\label{adjointdiff7}
\end{align}
and the state equation to (\ref{diff1})-(\ref{diff4}), which we get by differentiating the Lagrangian $\LL_2$ with respect to $z$, is given in strong form by
\begin{equation}
\td{y}-\mathrm{div}(k\nabla y)=0 \quad \text{in }\Omega \times (0,T]. \label{designdiff}
\end{equation}
We formulate explicitly the interface condition of (\ref{adjointle1})-(\ref{adjointle3}) and (\ref{adjointdiff1})-(\ref{adjointdiff7}) by
\begin{align}
\label{interface_w}\left\llbracket w\right\rrbracket & =0, & \hspace{-2cm} \left\llbracket \sigma(w)\cdot n\right\rrbracket =0& \hspace{1cm} \text{on }\Gi,\\[5pt]
\left\llbracket z\right\rrbracket & =0, & \hspace{-2cm} \left\llbracket k\nd{z}\right\rrbracket =0& \hspace{1cm}\text{on }\Gi\times[0,T).
\end{align}

There are a lot of options to prove shape differentiability of shape functionals, which depend on a solution of a PDE, and to derive the shape derivative of a PDE constrained shape optimization problem. The min-max approach \cite{Delfour-Zolesio-2001}, the chain rule approach \cite{SokoZol}, the Lagrange method of C\'{e}a \cite{Cea-RAIRO} and the rearrangement method \cite{Ito-Kunisch-Peichl} have to be mentioned in this context.
A nice overview about these approaches is given in \cite{Sturm}. 
The next subsection deals with the shape derivatives of the models above.


\subsection{Shape derivatives}
\label{subsection_shapederivatives}

In the sequel, we deduce the shape derivative of each shape functional $j_l$, where $l\in \{1,2,3,4\}$. The sum of these four shape derivatives is the shape derivative of $J$. Note that we need only the volume form of the shape derivative of $j_l$ with $l\in \{1,2,3\}$ with respect to the Steklov-Poincar\'{e} metric and the optimization techniques established in \cite{schulz2015Steklov,Welker}.
The existence of these shape derivatives is given by the theorem of Correa and Seeger \cite[theorem 2.1]{CorreaSeger}.

The following theorem gives a representation of the shape derivative of $j_1$ expressed as volume integral.
\begin{theorem}
Let $u\in H^1(\Omega,\mathbb{R}^d)$ denote the weak solution of (\ref{le1})-(\ref{le4}). Moreover, $w\in H^1(\Omega,\mathbb{R}^d)$ denote the weak solution of the adjoint equation (\ref{adjointle1})-(\ref{adjointle3}). Then the shape derivative of $j_1$ in direction $V$ is given by
\begin{equation}
\boxed{
\begin{aligned}
Dj_1(u,\Omega)[V] =\int_{\Omega} & -\sigma(u):\left(\nabla V^T\nabla w\right)-\sigma(w):\left(\nabla V^T \nabla u\right)\\ & +\text{\emph{div}}(V) \left(\nu_1\sigma(u):\epsilon(u)+\sigma(w):\nabla u\right)\hspace{.7mm}dx.
\end{aligned}
}
\label{sd_j1}
\end{equation}
\end{theorem}

\begin{proof}
We consider the Lagrangian $\LL_1$. In analogy to \cite[chapter 10, subsection 5.2]{Delfour-Zolesio-2001}, we can verify that 
\begin{equation}
\label{minmax}
j_1(u,\Omega)= \min_{u\in H^1(\Omega,\mathbb{R}^d)} \max_{w\in H^1(\Omega,\mathbb{R}^d)} \LL_1 (u,w,\Omega)
\end{equation}
holds.
We apply the theorem of Correa and Seeger on the right-hand side of (\ref{minmax}). The verification of the assumptions of this theorem can be checked in the same way as in \cite[chapter 10, subsection 6.4]{Delfour-Zolesio-2001}. 
 
Applying the rules for differentiating volume and surface integrals given in (\ref{der_domain_int}) and (\ref{der_boundary_int}) yields
\begin{equation}
\label{shape_der_1}
\begin{split}
& D\LL_1(u,w,\Omega)[V]\\
&=\int_{\Omega}\nu_1 D_m\left(\sigma(u):\epsilon(u)\right)+D_m\left(\sigma(u):\epsilon(w)\right)\\
&\hspace*{10mm}+\mathrm{div}(V)\left(\nu_1 \sigma(u):\epsilon(u)+\sigma(u):\epsilon(w)\right)dx\\
&\hspace*{4mm}-\int_{\Gi} D_m\left( \left\llbracket \left(\sigma(u)\cdot n\right)^Tw\right\rrbracket\right) + \mathrm{div}_{\Gi}(V)\left\llbracket \left(\sigma(u)\cdot n\right)^Tw \right\rrbracket ds  \\
&\hspace*{4mm}-\int_{\Go}D_m\left(\left(\sigma(u)\cdot n\right)^Tw\right) + \mathrm{div}_{\Go}(V)\left( \left(\sigma(u)\cdot n\right)^Tw\right) \hspace{.7mm}ds\\
&\hspace*{4mm}-\int_{\Gb}D_m\left(w_1^T u\right)+ \mathrm{div}_{\Gb}(V) w_1^T u \hspace{.7mm}ds \\
&\hspace*{4mm}-\int_{\Go\setminus\Gb}D_m\left(w_2^T \left(\sigma (u)\cdot n-f\right)\right)\\
&\hspace*{27mm}+ \mathrm{div}_{\Go\setminus\Gb}(V)\hspace{.7mm}w_2^T \left(\sigma (u)\cdot n-f\right)\hspace{.3mm}ds.
\end{split}
\end{equation}
%
Note that 
\begin{equation}
\label{umformung}
\sigma(u):\epsilon(w)=\sigma(u):\nabla w
\end{equation}
holds because $\sigma(u)$ is symmetric (cf.~\cite{Sewell}). Moreover, we have
\begin{align}
D_m(\sigma(u):\nabla w)=\mu D_m\left((\nabla u+\nabla u^T): \nabla w\right)+\lambda D_m\left(\text{div}(u)\text{div}(w)\right)
\end{align}
due to the definition of $\sigma$ and $\epsilon$.
Combining (\ref{product_rule}) with (\ref{material_grad}) in $\mathbb{R}^d$ yields
\begin{equation}
\label{div_umformung}
D_m(\text{div}(u))=D_m(I:\nabla u)=\text{div}(\dot{u})-I:(\nabla V^T\nabla u).
\end{equation}
By applying the product rule and appropriate transformations we get
\begin{equation}
\label{derivative_sigmaespilon}
\begin{split}
& D_m(\sigma(u):\epsilon(w))\\
&= \sigma(\dot{u}):\nabla w+\sigma(u):\nabla \dot{w}-\sigma(u):(\nabla V^T\nabla w)-\sigma(w):(\nabla V^T\nabla u)
\end{split}
\end{equation}
due to (\ref{material_grad}) and (\ref{umformung})-(\ref{div_umformung}).
Since the outer boundary $\Go$ is fixed, we can choose the deformation vector field $V$ equals zero in small neighbourhoods of $\Go$. Moreover, each material derivative in small neighbourhoods of $\Go$ is equal to zero. Thus, the outer integrals in (\ref{shape_der_2}) are not further considered.
From (\ref{shape_der_1}) we obtain the following under consideration of the adjoint and design equation in weak form and by appropriate transformations:
\begin{equation}
\label{shape_der_2}
\begin{split}
 D\LL_1(u,w,\Omega)[V]
&=\int_{\Omega}-\sigma(u):\left(\nabla V^T\nabla w\right)-\sigma(w):\left(\nabla V^T \nabla u\right)\\
&\hspace*{10mm}+\mathrm{div}(V)\left(\nu_1\sigma(u):\epsilon(u)+\sigma(w):\nabla u\right)\hspace{.7mm}dx\\
&\hspace*{4mm}+\int_{\Gi}  \left\llbracket \left(\sigma(w)\cdot n\right)^T\dot{u} - D_m\left(\left(\sigma(u)\cdot n\right)^T\right)w\right\rrbracket \\ 
&\hspace*{15mm}+ \mathrm{div}_{\Gi}(V)\left\llbracket \left(\sigma(u)\cdot n\right)^Tw \right\rrbracket ds 
\end{split}
\end{equation}
Due to the identity
$\left\llbracket \Phi \Psi\right\rrbracket=\left\llbracket \Phi\right\rrbracket \Psi_\text{out} +\Phi_\text{int} \left\llbracket \Psi\right\rrbracket= \Phi_\text{out} \left\llbracket \Psi\right\rrbracket+\left\llbracket \Phi\right\rrbracket \Psi_\text{int}$, 
which implies
\begin{equation*}
\left\llbracket \Phi \Psi\right\rrbracket=0 \text{ if } \left\llbracket \Phi\right\rrbracket=0\wedge \left\llbracket \Psi\right\rrbracket=0,
\end{equation*}
and the interface conditions (\ref{interface_cond1}) and (\ref{interface_w}), the interface integral vanishes in (\ref{shape_der_2}). By applying the theorem of Correa and Seeger, we obtain (\ref{sd_j1}).
\end{proof}

Now, we consider the shape functional $j_2$ and the diffusion problem (\ref{diff1})-(\ref{diff4}). If we consider $\LL_2$, its shape derivative can be deduced analogously to the computations in \cite{schulz2014structure}.
Thus, we get
\begin{equation}
\boxed{
\begin{aligned}
Dj_2(y,\Omega)[V] =\int_{0}^{T}\hspace{-.1cm}\int_{\Omega} & -k\nabla y^T\left(\nabla V+\nabla V^T\right)\nabla z\\ &+\text{div}(V) \left( \frac{\nu_2}{2}(y-\bar{y})^2 + \td{y}z+k\nabla y^T\nabla z\right)\hspace{.7mm}dx\hspace{.5mm}dt.
\end{aligned}
}
\label{sd_j2}
\end{equation}

The shape derivative of $j_3$, which is responsible for a space-filling design of $\Omega_\text{int}$, is given by
\begin{equation}
\boxed{
Dj_3(\Omega)[V] =\nu_3\int_{\Omega_{\text{out}}} \text{div}(V) \hspace{.7mm}dx.\label{sd_j3}
}
\end{equation}
It results directly from an application of (\ref{der_domain_int}).
 
Moreover, the objective function $J$ includes the perimeter regularization term $j_4$.
We get the shape derivative of this regularization term by applying (\ref{der_boundary_int}):
\begin{align*}
Dj_4(\Omega)[V] & =\nu_4\int_{\Gi}\text{div}_\Gi(V) \hspace{.7mm}ds=\nu_4\int_{\Gi}\text{div}_\Gi(\left<V,n\right>n) \hspace{.7mm}ds\\
&= \nu_4\int_{\Gi}\left<V,n\right>\text{div}_\Gi(n) \hspace{.7mm}ds,
\end{align*}
i.e.,
\begin{equation}
\boxed{
Dj_4(\Omega)[V] =\nu_4\int_{\Gi}\kappa\left<V,n\right> ds,\label{sd_j4}
}
\end{equation}
where $\kappa:=\text{div}_\Gi(n)$ denotes the mean curvature of $\Gi$. 

So far, we have deduced derivative. However, in order to optimize on shape spaces, we need the gradient with respect to the inner product under consideration. In the next section, we comment on shape spaces and metrics.

\section{Optimization based on Steklov-Poincar\'{e} metrics}\label{section_algorithm}

As pointed out for example in \cite{schulz2015Steklov,schulz2014structure,Welker}, PDE constrained shape optimization problems can be embedded in the framework of optimization on shape spaces. This section summarizes the way from shape derivatives to an entire optimization algorithm in a suitable shape space.

In our setting, we think of shapes as boundary contours of deforming objects. 
We consider so-called prior shapes $\Gamma_0$. These are boundaries $\Gamma_0=\partial\mathcal{X}_0$ of connected and compact subsets $\mathcal{X}_0\subset \Omega\subset\R^d$ with $\mathcal{X}_0\neq\emptyset$, where $\Omega$ denote bounded Lipschitz domains. Let $\mathcal{X}_0$ be Lipschitz domains and called prior sets.
Shapes -- in our setting -- arise from $H^1$-deformations of such prior sets $\mathcal{X}_0$. These $H^1$-deformations, evaluated at a prior shape $\Gamma_0=\partial \mathcal{X}_0$, give deformed shapes $\Gamma$ if the deformations are injective and continuous. These shapes are called of class $H^{1/2}$ and proposed firstly in \cite{schulz2015Steklov}. They are defined by
\begin{equation}\label{shape_mainifold}
{\cal B}^{1/2}(\Gamma_0,\R^d):=
{\cal H}^{1/2}(\Gamma_0,\R^d)\big\slash \mbox{Homeo}^{1/2}(\Gamma_0),
\end{equation}
where ${\cal H}^{1/2}(\Gamma_0,\R^d)$ is given by
\begin{equation}
\label{force-ball}
\begin{split}
{\cal H}^{1/2}(\Gamma_0,\R^d):=
\{w\colon \Gamma_0\to \Omega \colon & \exists W\in H^1(\Omega,\Omega) \text{ s.t.}\\ & W\hspace{-.5mm}\,\rule[-2mm]{.1mm}{4mm}_{\hspace{.5mm}\Gamma_0} \text{ injective, continuous}\text{, }W\hspace{-.5mm}\,\rule[-2mm]{.1mm}{4mm}_{\hspace{.5mm}\Gamma_0}=w\}
\end{split}
\end{equation}
and $\mbox{Homeo}^{1/2}(\Gamma_0)$ is defined by
\begin{equation}
\mbox{Homeo}^{1/2}(\Gamma_0)\\:=
\{ w\colon w\in {\cal H}^{1/2}(\Gamma_0,\R^d)\text{, }w\colon \Gamma_0\to \Gamma_0 \text{ homeomorphism}\}.
\end{equation}
In the following, we assume that ${\cal B}^{1/2}\left(\Gamma_0,\R^d\right)$ has a manifold structure. If necessary, we can refine the space ${\cal B}^{1/2}(\Gamma_0,\R^d)$ as already described in \cite{Welker}. However, this conceivable limitation leaves the following theory untouched. If $\Gamma\in{\cal B}^{1/2}(\Gamma_0,\R^d)$ is smooth enough to admit a normal vector field $n$, the following isomorphisms arise from definition (\ref{force-ball}):
\begin{equation}
\label{isomophism}
\begin{split}
 T_\Gamma {\cal B}^{1/2}(\Gamma_0,\R^d) & \cong
\{h\colon  h=\phi n  \text{ a.e.}\text{, } \phi\in H^{1/2}(\Gamma) \text{ continuous} \}
\\ & \cong
\{\phi\colon \phi\in H^{1/2}(\Gamma) \text{ continuous} \}
\end{split}
\end{equation}

We consider the following scalar products, the so-called \emph{Steklov-Poincar\'{e} metrics} (cf.~\cite{schulz2015Steklov}):
\begin{equation}\label{scp}
\begin{split}
g^S\colon H^{1/2}(\Gamma_\text{int})\times H^{1/2}(\Gamma_\text{int}) & \to \R,\\
(\alpha,\beta) &\mapsto \langle\alpha,(S^{pr})^{-1}\beta\rangle=
\int_{\Gamma_\text{int}} \alpha(s)\cdot [(S^{pr})^{-1}\beta](s)\ ds
\end{split}
\end{equation}
Here $S^{pr}$ denotes the projected Poincar\'e-Steklov operator which is given by
\begin{equation}
S^{pr}\colon  H^{-1/2}(\Gamma) \to H^{1/2}(\Gamma),\
         \alpha \mapsto (\gamma_0 U)^T n,
\end{equation}
where $\gamma_0\colon  H^1_0(\Omega,\R^d) \to H^{1/2}(\Gamma,\R^d)$, $U \mapsto U\,\rule[-2mm]{.1mm}{4mm}_{\, \Gamma}$ and $U\in H^1_0(\Omega,\R^d)$ solves the Neumann problem
\begin{equation}\label{weak-elasticity-N2}
a(U,V)=\int_{\Gamma} \alpha\cdot (\gamma_0 V)^T n\ ds\quad \forall\hspace{.3mm}  V\in H^1_0(\Omega,\R^d)
\end{equation}
with $a$ being a symmetric and coercive bilinearform.

We are now able to formulate an optimization algorithm on the shape space ${\cal B}^{1/2}\left(\Gamma_0,\R^d\right)$ with respect to $g^S$. Before we can do that, we have to state its connection to shape calculus. Due to the Hadamard structure theorem (cf.~\cite[theorem 2.27]{SokoZol}), there exists a scalar distribution $r$ on the boundary of the domain under consideration. If we assume $r\in L^1(\Gamma)$ and $V\,\rule[-2mm]{.1mm}{4mm}_{\hspace{.6mm}\Gamma}=\alpha n$, the shape derivative can be expressed as the boundary integral 
\begin{equation}
\label{HadamardConcisely}
DJ_{\Gamma}[V]=\int_{\Gamma} \alpha \hspace{.5mm} r \ ds.
\end{equation}
Due to the isomorphism (\ref{isomophism}) and the expression (\ref{HadamardConcisely}), we can state the connection of $\mathcal{B}^{1/2}(\Gamma_0,\R^d)$ with respect to the Steklov-Poincar\'{e} metric $g^S$ to shape calculus.
The distribution $r$ is often called the \emph{shape gradient}. This is confusing, because one has to note that gradients always depend on chosen scalar products defined on the space under consideration. It rather means that $r$ is the usual $L^2$-shape gradient. However, we have to find a representation of the shape gradient with respect to $g^S$.
Such a representation $h\in T_{\Gamma} {\cal B}^{1/2}(\Gamma_0,\R^d) \cong \{h\colon h\in H^{1/2}(\Gamma) \text{ continuous} \}$ is determined by
\begin{equation}
\label{connection1}
g^S(\phi,h)=\left(r,\phi\right)_{L^2(\Gamma)},
\end{equation}
which is equivalent to
\begin{equation}
\label{connection2}
\int_{\Gamma} \phi(s)\cdot [(S^{pr})^{-1}h](s) \ ds=\int_{\Gamma} r(s)\phi(s) \ ds
\end{equation}
for all continuous $\phi\in H^{1/2}(\Gamma)$.
Based on the connection (\ref{connection1}) we can formulate an algorithm to solve PDE constrained shape optimization problems on $\mathcal{B}^{1/2}(\Gamma_0,\R^d)$ with respect to $g^S$. From (\ref{connection2}) we get $h=S^{pr}r=(\gamma_0 U)^T n$, where $U\in H^1_0(\Omega,\R^d)$ solves
\begin{equation}
\label{main}
a(U,V)=\int_{\Gamma} r\cdot (\gamma_0 V)^T n \ ds
=DJ_{\Gamma}[V]=DJ_\Omega[V] \quad \forall\hspace{.3mm}  V\in H^1_0(\Omega,\R^d).
\end{equation}
This means that we get the gradient representation $h$ and the mesh deformation $U$ all at once and that we have to solve 
\begin{equation}
a(U, V) = b(V)
\label{deformatio_equation}
\end{equation}
for all test functions $V$ in the optimization algorithm,
where $b$ is a linear form and given by
$$b(V):=DJ_\text{vol}(\Omega)[V]+DJ_\text{surf}(\Omega)[V].$$
$J_\text{surf}(\Omega)$ denotes parts of the objective function leading to surface shape derivative expressions, e.g., perimeter regularizations. The shape derivative $DJ_\text{surf}(\Omega)[V]$ of these terms are incorporated as Neumann boundary conditions. Parts of the objective function leading to volume shape derivative expressions are denoted by $J_\text{vol}(\Omega)$. 
However, note that from a theoretical point of view the volume and surface shape derivative formulations have to be equal to each other for all test functions. Thus, $DJ_\text{vol}[V]$ is assembled only for test functions $V$ whose support includes $\Gamma$, i.e.,
$$DJ_\text{vol}(\Omega)[V]=0 \quad \forall\hspace{.3mm}  V \text{ with } \text{supp}(V)\cap\Gamma=\emptyset.$$
We call (\ref{deformatio_equation}) the \emph{deformation equation}. 
The entire optimization algorithm is given in figure \ref{algorithm} and explained in detail in the next section. 

\tikzset{mybox/.style={
rectangle,
rounded corners=2mm,
thick,
draw=black,
text width=30em,
text centered,
drop shadow,fill=white}
}
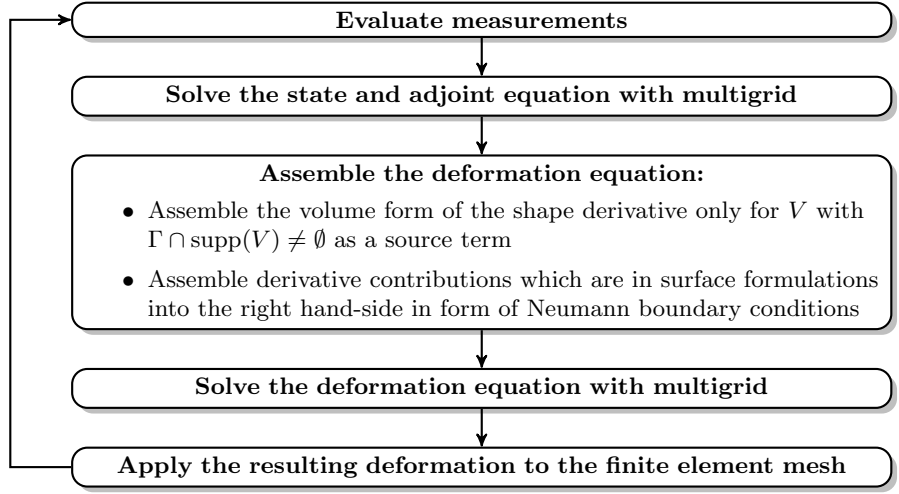
\begin{figure}
\begin{center}
\begin{tikzpicture}[node distance=.5cm,>=stealth',bend angle=45,auto]
\tikzstyle{every node}=[font=\small]
\node[mybox] (interp) {\textbf{Evaluate measurements}};
\node[mybox, below = of interp] (primal) {\textbf{Solve the state and adjoint equation with multigrid}};
\node[mybox, below = of primal] (grad) {\textbf{Assemble the deformation equation:}\\[.1cm]
\begin{itemize}
\item Assemble the volume form of the shape derivative only for $V$ with $\Gamma\cap \mathrm{supp}(V)\not=\emptyset$ as a source term
\item Assemble derivative contributions which are in surface formulations into the right hand-side in form of Neumann boundary conditions
\end{itemize}};
\node[mybox, below = of grad] (Le1) {\textbf{Solve the deformation equation with multigrid}};
\node[mybox, below = of Le1] (Le2) {\textbf{Apply the resulting deformation to the finite element mesh}};
\path (interp) edge[->,thick] (primal);
\path (primal) edge[->,thick] (grad);
\path (grad) edge[->,thick] (Le1);
\path (Le1) edge[->,thick] (Le2);
\draw [->,thick] ($ (Le2.west) $) -- ++(-0.8,0) -| ($ (interp.west) - (0.8,0) $) -- ($ (interp.west) $) ;
\end{tikzpicture}
\caption{Optimization algorithm}
\label{algorithm}
\end{center}
\end{figure}

In the setting of section \ref{section_model}, each $\Gamma_i$ is a shape in the above sense, i.e., an element of $\mathcal{B}^{1/2}\left(\Gamma_0,\mathbb{R}^d\right)$.
The volume formulation is given by
$$DJ_\text{vol}(\Omega)[V]= Dj_1(\Omega)[V]+Dj_2(\Omega)[V]+Dj_3(\Omega)[V]$$
and the surface formulation by
$$DJ_\text{surf}(\Omega)[V]=Dj_4(\Omega)[V].$$

The next section is devoted to the numerical realization of algorithm \ref{algorithm} in order to solve the PDE constrained model problems given above.

\section{Numerical results}\label{section_results}

We focus on three numerical experiments, which are selected in order to demonstrate challenges arising for large-scale multigrid shape optimizations.
All results involved are computed at the High Performance Computing Center Stuttgart using the machine HAZELHEN with up to $16\,384$ cores.

In the \textit{first experiment}, we choose $\nu_1=\nu_2=0$ such that we end up with a pure geometric optimization problem without any PDE constraints.
Hereby, we want to demonstrate a proper choice of boundary conditions in the mesh deformation in order to obtain results for periodic domains.
This optimization problem is motivated by the question for a space-filling cell design with least area of surfaces and only one type of cells.
It goes back to the 19th century and is also known as the Kelvin problem, for which he proposed a solution (cf.~\cite{thomson1887division}) based on cells as depicted in figure \ref{fig_volume_only_optimal}.
His conjecture was disproved by a counter example in \cite{weaire1994counter}.
However, these authors propose a solution with two types of cells.
The resulting optimal shape might be understood as the result of a biological growing process and used as a building block for finite element models of the human skin (cf.~\cite{muha2011effective}).
\begin{figure}
\centering
\begin{subfigure}{.5\textwidth}
  \centering
  \includegraphics[width=0.8\linewidth]{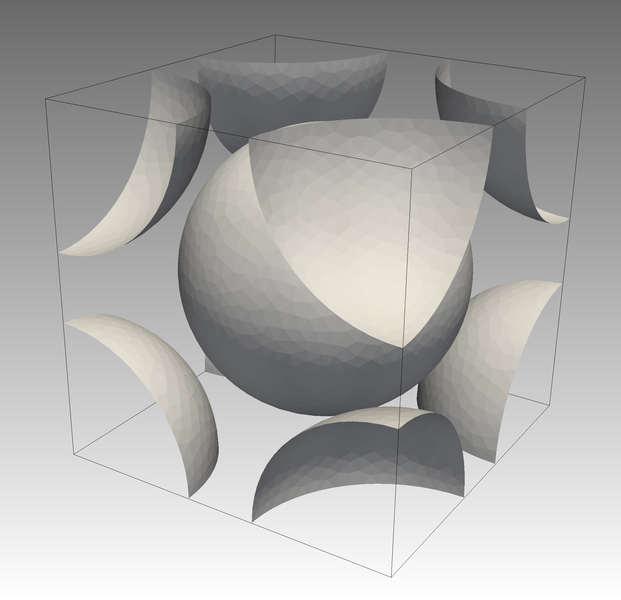}
  \caption{Initial configuration with symmetries}
  \label{fig_volume_only_initial}
\end{subfigure}%
\begin{subfigure}{.5\textwidth}
  \centering
  \includegraphics[width=0.8\linewidth]{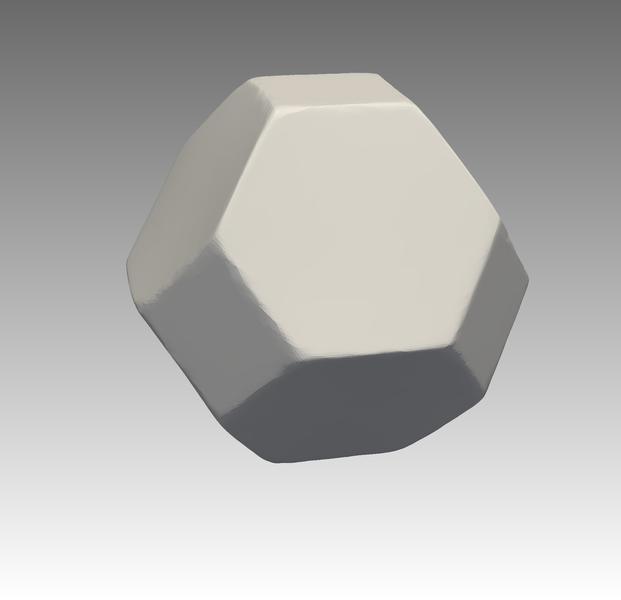}
  \caption{Optimal shape}
  \label{fig_volume_only_optimal}
\end{subfigure}
\caption{Approximation of a Kelvin cell in a symmetric domain}
\label{fig_volume_only}
\end{figure}
Figure \ref{fig_volume_only_initial} depicts the initial geometry for the optimization, where we want to exploit symmetries on all outer surfaces.
The shape derivative of the objective \eqref{objective} for this particular case is given as the sum of the derivatives \eqref{sd_j3} and \eqref{sd_j4} with $\nu_3 = \nu_4 = 1.0$.
Note that this results in a mixture of volume and surface formulations.

\begin{figure}
\centering
\begin{subfigure}{.5\textwidth}
  \centering
  \includegraphics[width=0.9\linewidth]{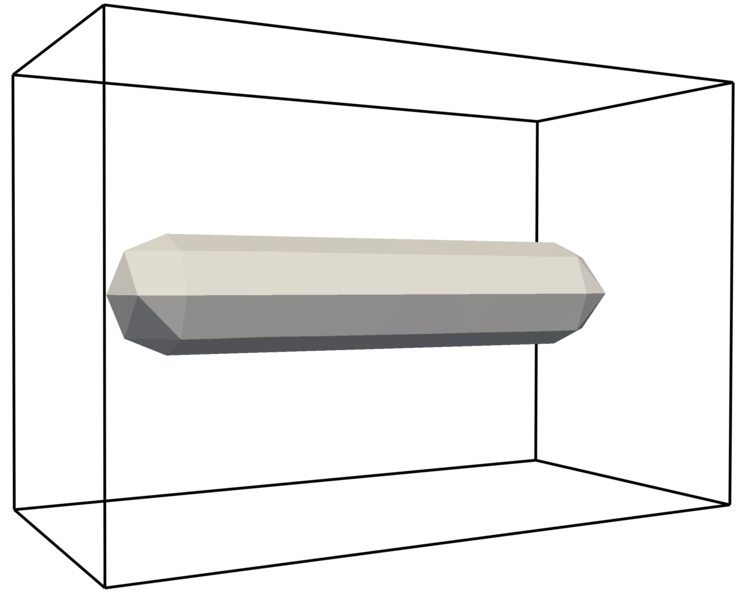}
  \caption{Inital shape}
  \label{fig_initial_tube}
\end{subfigure}%
\begin{subfigure}{.5\textwidth}
  \centering
  \includegraphics[width=0.9\linewidth]{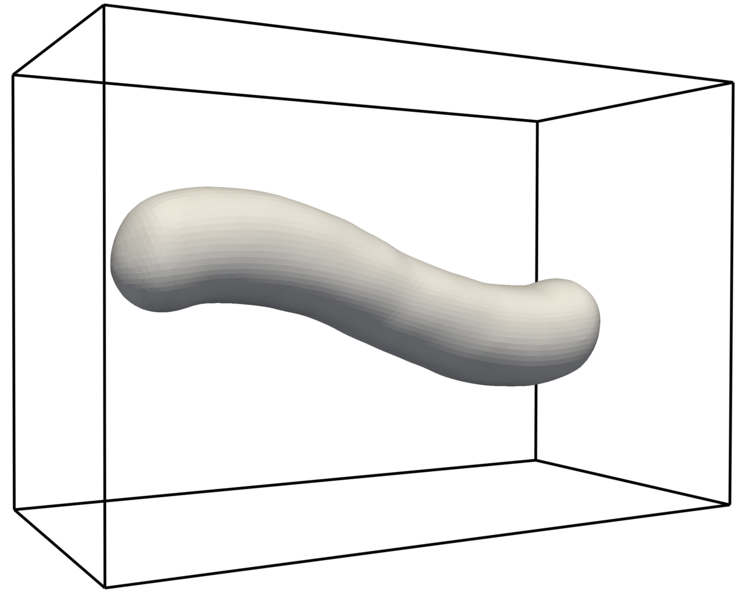}
  \caption{Optimal shape}
  \label{fig_optimal}
\end{subfigure}
\caption{Sharp-edged initial geometry with identical coarse and fine grid under 6 levels of hierarchical refinements and optimal solution on finest level}
\label{fig_tube_optim}
\end{figure}
\begin{figure}
\includegraphics[width=1.0\textwidth]{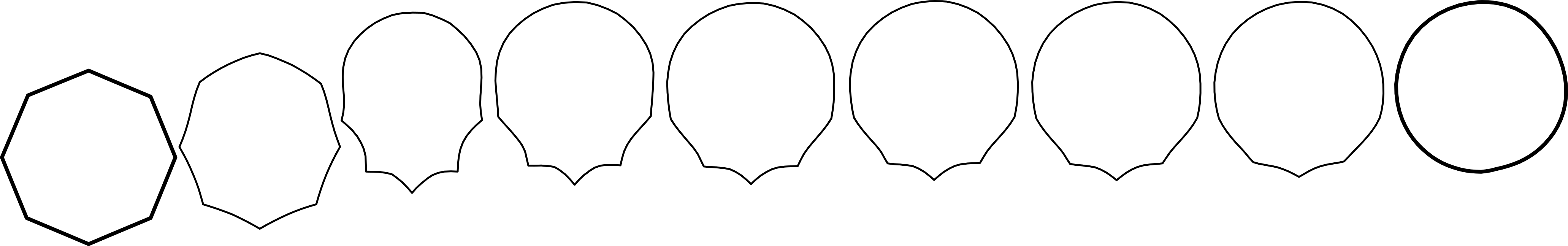}
\caption{Cross-section of the tube at $25\%$ of the length for the first iterations with perimeter regularization}
\label{fig_tube_multigrid}
\end{figure}

The linear and bilinear form defining the shape metric \eqref{deformatio_equation}, which gives also the information for the mesh deformation field $U$, is chosen to be the weak form of a linear elasticity model.
In order to model the periodicity of the domain, we choose the following Dirichlet and Neumann conditions at the outer boundaries depending on the components $U = (U_x, U_y, U_z)^T$:
\begin{align} 
U_x & = 0,\quad \sigma\left( (0, U_y, U_z)^T \right) \cdot n  = 0 \qquad \text{on}\; \Gfront \cup \Gback \label{eq_def_2} \\ 
U_y & = 0,\quad \sigma\left( (U_x, 0, U_z)^T \right) \cdot n  = 0 \qquad \text{on}\; \Gl \cup \Gr \label{eq_def_3} \\ 
U_z & = 0,\quad \sigma\left( (U_x, U_y, 0)^T \right) \cdot n  = 0 \qquad \text{on}\; \Gb \cup \Gt \label{eq_def_4} 
\end{align}
This particular choice ensures that nodes are allowed to slide within $\Go$, but not to leave the planes.
The Lam\'e parameter are fixed to $\lambda = 0.01$ and $\mu = 0.1$.
We apply this shape metric to all numerical examples in this sections, due to the fact that the sliding conditions minimize the influence of the outer boundary on the optimization of $\Gi$.
Especially, when $\Gi$ intersects $\Go$ like in the volume experiment shown in figure \ref{fig_volume_only}, the sliding effect at boundaries is obligatory. 

The result of the optimization is visualized in figure \ref{fig_volume_only_optimal}.
It can be identified as an approximation to a truncated octahedron according to Kelvins conjecture.
This shape is also denoted as a tetrakaidecahedron with 6 square faces and 8 hexagonal faces.
We obtain this result after 12 gradient steps on a fine grid with approximately $3.9\cdot 10^6$ and a coarse grid with $60\,533$ finite elements.

Throughout this section all linear systems are solved with the multigrid preconditioned conjugated gradient solver of the software PETSC.
Here a symmetric Gauss-Seidel smoother is applied and the coarse grid solver is a direct factorization computed with the SUPERLU-DIST library.
This choice is possible, since all PDEs to be solved, i.e., the parabolic diffusion equation and linear elasticity, are discretized as symmetric, positive definite matrices.
Moreover, we choose linear finite elements on tetrahedral grids for the discretization of all PDE models involved.
Load balancing is performed with the PARMETIS graph partitioning library.
As mentioned above, it is sufficient to partition and balance only with respect to tetrahedrons and not also surface triangles at $\Gi$.
This is due to the fact that we plug in the volume form of shape derivatives whenever possible.
Thus, there are only very few surface-only operations, which do not dramatically affect the scalability.
By solving the deformation equation \eqref{deformatio_equation} we then obtain a deformation field $U$.
This vector field is added as a deformation to all nodes in the finite element meshes on all multigrid levels.
This means that the fine grid solution $U$ is interpolated at the nodes of the coarser grids.
The multigrid hierarchy is designed such that for the initial geometry the nodes of the coarse grid are a subset of the nodes of the fine grid.
Thus, this property is maintained throughout the whole optimization process.
In principle, this means that the shape optimization takes place at the fine grid and the coarse grids are carried along during this process for the preconditioner of the linear solver.

The \textit{second numerical test} is chosen in order to touch upon the main topic of this paper, namely the interaction of multigrid solvers and shape optimization.
We are in a simplified situation of the optimization problem \eqref{eq_minimization}, which is only subject to the diffusion model \eqref{diff1}-\eqref{diff4}.
Thus, the objective function is of tracking type with regularization corresponding to $\nu_1 = \nu_3 = 0$.
The domain $\Omega = \Omega_\text{out} \cupdot \Omega_\text{int}$ has only one inclusion with a jumping diffusion coefficient $k_\text{out} = 1.0$ and $k_\text{int} = 0.001$.
Figure \ref{fig_tube_optim} shows this situation.
The initial geometry is depicted on the left hand side and the target is shown on the right hand side.
Data measurements $\bar{y}$ are generated in terms of the same model equation and a similar shape to the one in figure \ref{fig_optimal}.
We assume two points in time for measurements -- one at the final time $T=15$ and one after $7.5$ seconds.

\begin{figure}
\centering
\begin{subfigure}{.5\textwidth}
  \centering
  \includegraphics[width=1.0\linewidth]{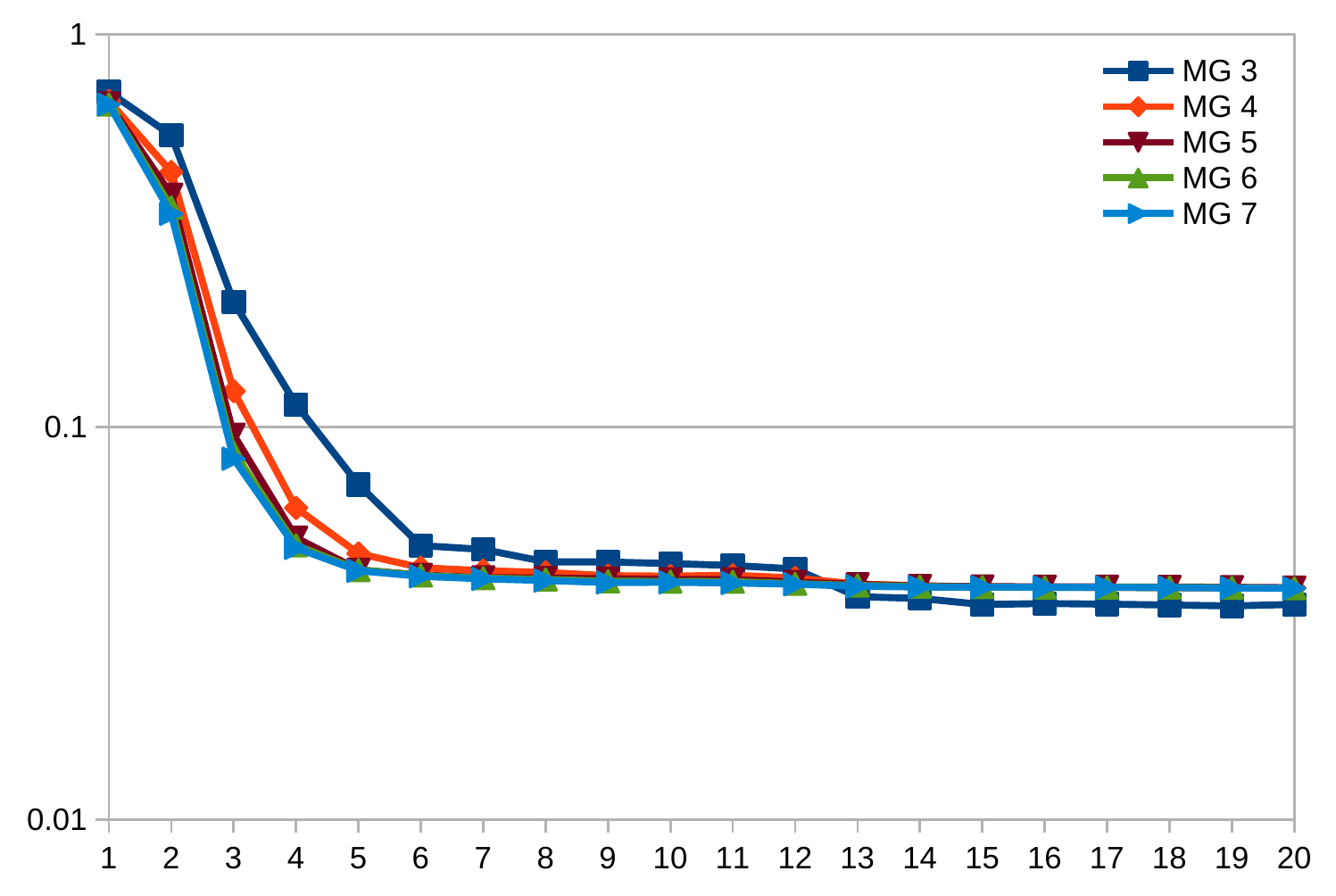}
  \caption{Objective value}
  \label{fig_objective}
\end{subfigure}%
\begin{subfigure}{.5\textwidth}
  \centering
  \includegraphics[width=1.0\linewidth]{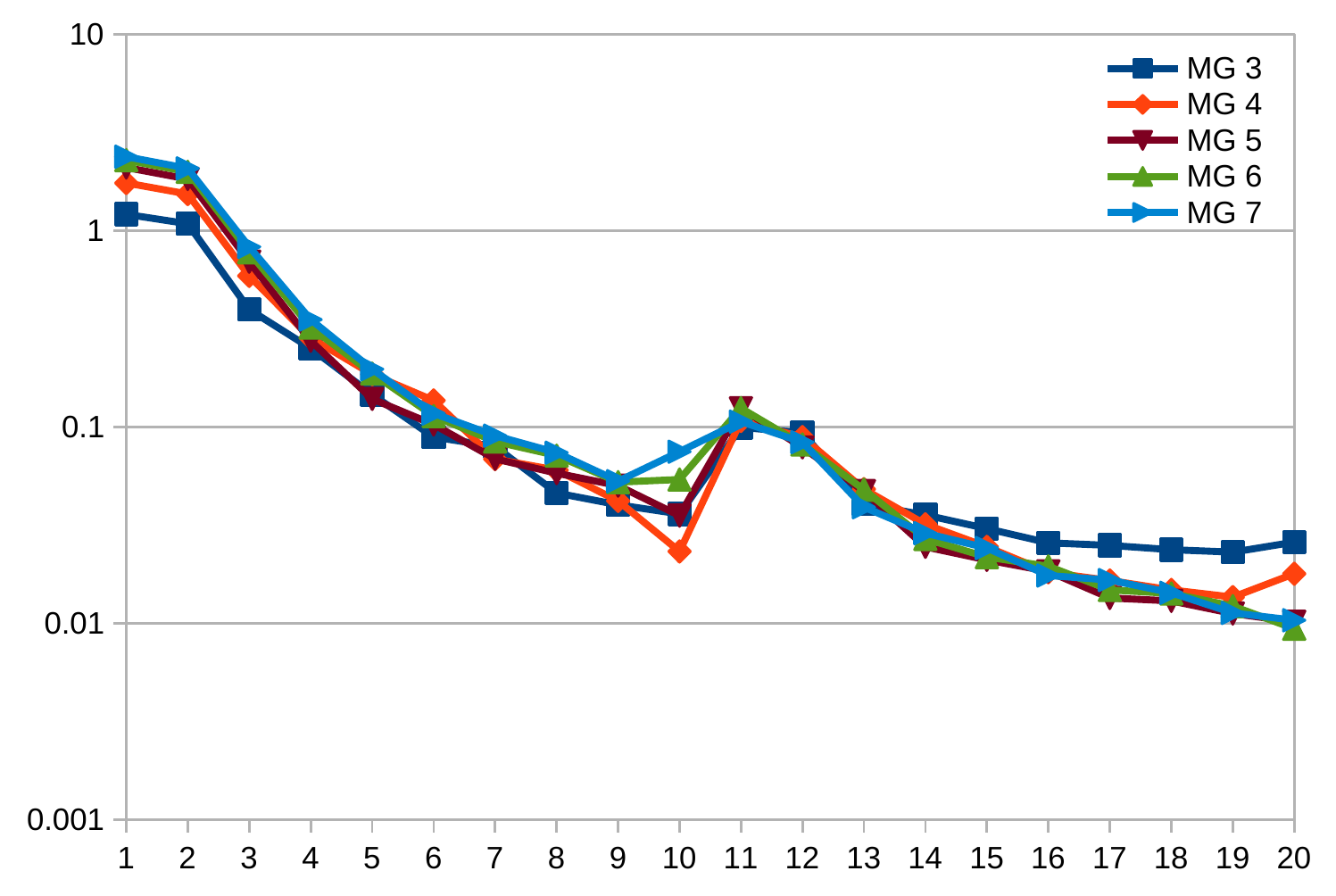}
  \caption{Gradient in $g^S$ norm}
  \label{fig_gradient}
\end{subfigure}
\caption{Objective value and gradient in $g^S$ norm for the pure parabolic test case with regularization in the first 10 iterations}
\label{fig_objecive_and_gradient}
\end{figure}

Measurements are assumed to be represented within 10\,000 Gaussian type radial basis functions.
This involves expensive but necessary computations, since $\bar{y}$ has to be evaluated on varying meshes during the optimization.
A general mesh to mesh interpolation on a distributed memory computer can not be computed efficiently, which makes a detour via radial basis functions to the expensive yet scalable method of choice.

The optimization problem is discretized using 7 levels of hierarchical mesh refinements starting from a coarse grid with $3\,390$ tetrahedral elements to a fine grid with approximately $8.89\cdot 10^8$.
For the diffusion we choose a backward Euler time stepping with a step-size of $\Delta t = 1.5$.
In figure \ref{fig_initial_tube}, we observe sharp edges in the geometry, which, in principle, should resolve a tube with radius $0.5$ and length $5$.
This is due to the hierarchical grid structure, which is typical for multigrid methods, since coarse and fine grid have to resolve the same geometry.
In our particular numerical test, this property of the mesh hierarchy together with the shape metric lead to problems in the discretization.

Figure \ref{fig_tube_multigrid} shows a cross-section of $\Omega_\text{int}$ at $25\%$ of the tubes length during the first iterates of the optimization.
We observe that edges, which have a smaller influence on the objective function than planes, tend to remain at their initial position.
This effect is intensified by the choice of the shape metric \eqref{deformatio_equation}.
Hereby, the impact of the shape derivative \eqref{sd_j2} on the geometry can be understood as a traction on the boundary, which is stronger on large planes compared to sharp edges.
The application of a limited memory BFGS update techniques also influences this behaviour, since relatively large steps are chosen especially in the first iterations.
\begin{figure}
\center
\includegraphics[width=0.6\textwidth]{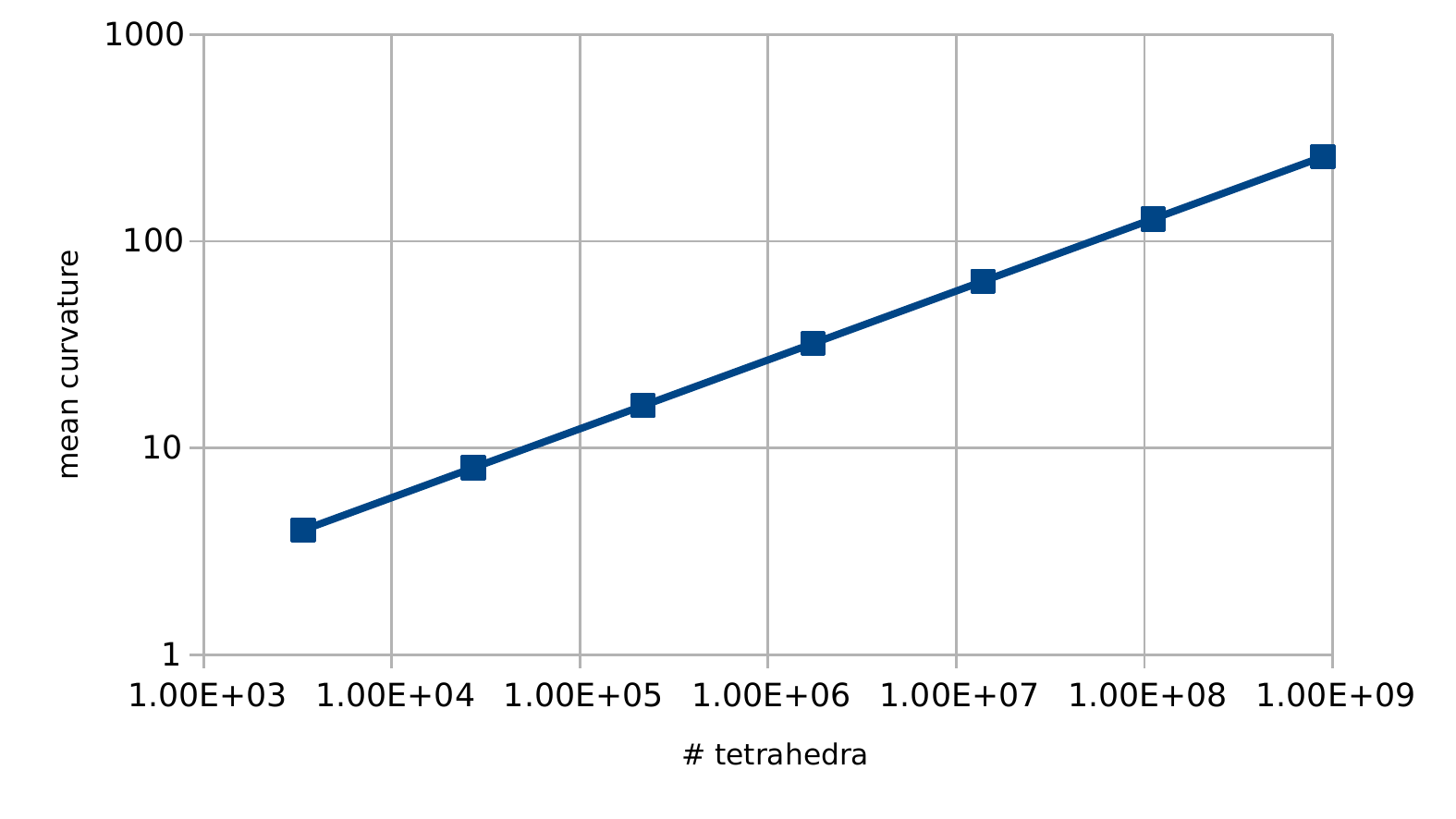}
\caption{log-log-plot of $L^\infty$-norm of mean curvature with respect to number of tetrahedral finite elements}
\label{fig_tube_curvature}
\end{figure}
\begin{figure}[h]
\centering
\begin{subfigure}{.5\textwidth}
  \centering
  \includegraphics[width=0.9\linewidth]{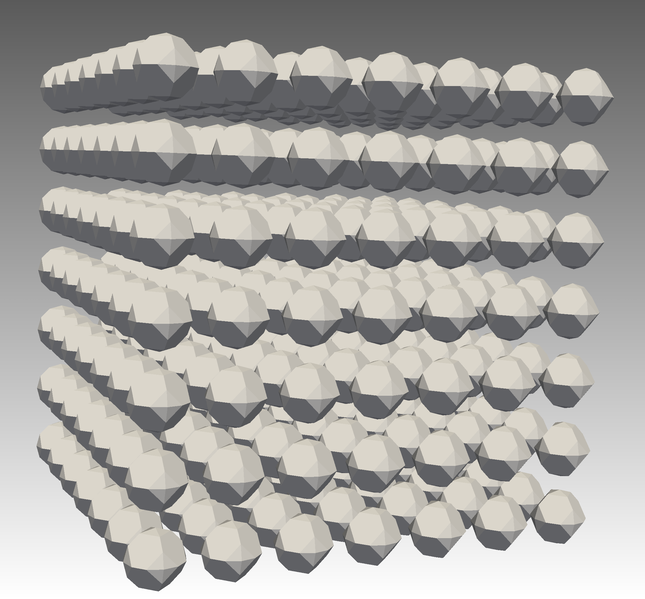}
  \caption{Initial grid}
  \label{fig_initial_grid}
\end{subfigure}%
\begin{subfigure}{.5\textwidth}
  \centering
  \includegraphics[width=0.9\linewidth]{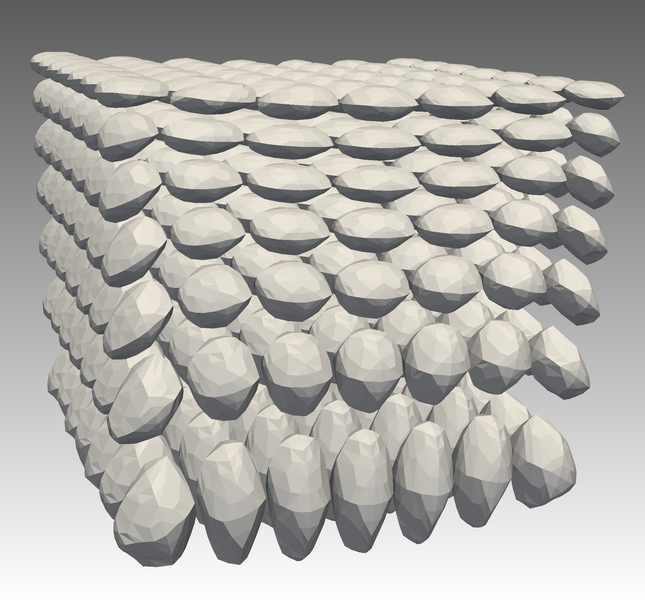}
  \caption{10 gradient steps}
  \label{fig_intermediate1_grid}
\end{subfigure}
\begin{subfigure}{.5\textwidth}
  \centering
  \includegraphics[width=0.9\linewidth]{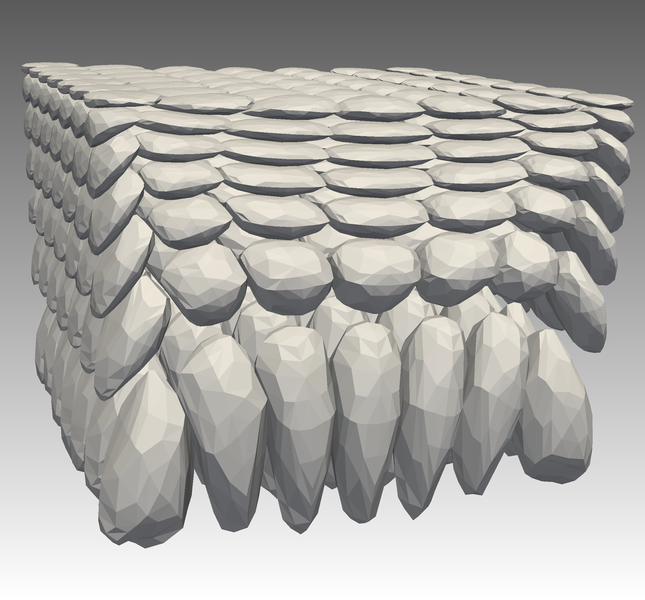}
  \caption{20 steps}
  \label{fig_intermediate2_grid}
\end{subfigure}%
\begin{subfigure}{.5\textwidth}
  \centering
  \includegraphics[width=0.9\linewidth]{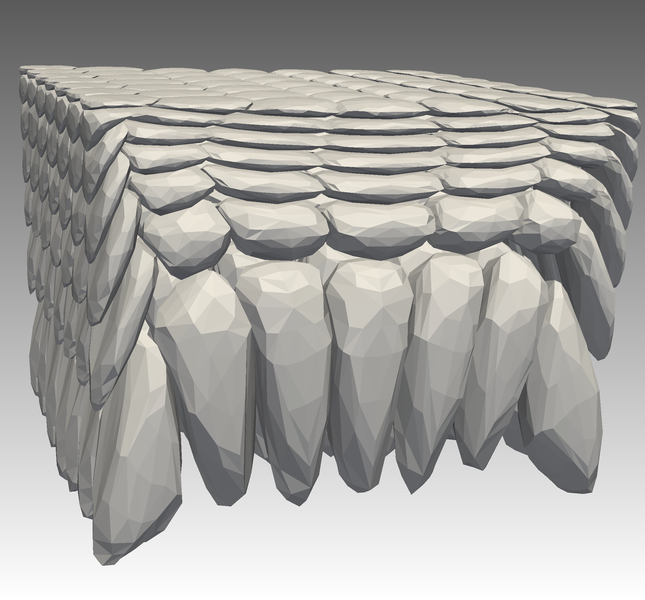}
  \caption{30 steps}
  \label{fig_final_grid}
\end{subfigure}
\caption{Optimization with respect to diffusion data measurements and elastic compliance of a cellular structure}
\label{fig_compliance_parabolic_opt}
\end{figure}

In figure \ref{fig_tube_multigrid}, it can be observed that plain optimization problem tends to overlapping elements.
This would lead to an immediate break down of the PDE solver.
In order to prevent this, the problem is regularized by the perimeter term $j_4$ with $\nu_4=0.01$ in \eqref{objective}.
The effect of the perimeter regularization can be observed throughout the figures.
As soon as sharp edges are smoothed out by the influence of the perimeter regularization, we can switch off this regularization by setting $\nu_4=0$.
In our particular case, we do this after 10 quasi-Newton steps.
The effect of this strategy is visualized in figure \ref{fig_objecive_and_gradient}.
The objective value is depicted on the left hand side, where there is a small step when the regularization is switched off.
Note that this step is a bit delayed, since we choose a limited memory BFGS strategy with a storage of 5 gradients.
The figure on the right hand side shows the $g^S$-norm of the gradient, where we observe a jump after 10 iterations.
The two figures in \ref{fig_objecive_and_gradient} can be seen as an indicator for mesh-independent convergence as expected due to the quasi-Newton method.
We can thus conclude, that for this particular test, a regularization is only required due to the multigrid approach.
Otherwise, if there are no fine-resolved kinks in the initial geometry, one could be successful without regularization.

We should mention that there is a formulation according to \cite[proposition 2.50]{SokoZol}, which is equivalent to (\ref{sd_j4}), for the derivative of the perimeter regularization given by
\begin{equation}
Dj_4(\Omega)[V] =\nu_4\int_{\Omega} \text{div}(V) - \left<\frac{\partial V}{\partial n},n\right> ds.\label{sd_j4_vol}
\end{equation}
This is attractive from a computational point of view, since the evaluation of $\kappa$ in each iteration is a surface-only operation.
Thus, its scalability is affected by the load balancing, which one would have to perform with respect to both volume and surface elements.
However, in our experiments, it seems that for the discretized problem the formulation \eqref{sd_j4} is more successive in eliminating kinks arising due to the multigrid mesh hierarchy.
Figure \ref{fig_tube_curvature} shows the approximation to the mean curvature on the 6 multigrid levels for the initial geometry.
The curvature computation follows the approach presented in \cite{meyer2003discrete}.
An exponential growth of the curvature over the grid levels can be observed, which makes it necessary to adapt $\nu_4$ depending on the finest grid level throughout the refined computations presented in figure \ref{fig_objecive_and_gradient}.

We now concentrate on the \textit{third numerical example}, which is the optimization with respect to the complete system \eqref{le1}-\eqref{diff4} gathering diffusion and elasticity.
The objective, in our particular case, is then formed by the constants $\nu_1 = 0.15$, $\nu_2 = 0.1$, $\nu_3 = 16.0$ and $\nu_4 = 0.01$.
Furthermore, data measurements $\bar{y}$ are chosen to be the simulated values $y$ with the initial geometry.
These are again represented in radial basis functions in order to interpolate them on arbitrary meshes.
The first challenge encountered here is that the coarse grid has to resolve the desired geometry as depicted in figure \ref{fig_compliance_parabolic_opt}.
This limits the coarseness, since each inclusion, that should be simulated, has to be present in the entire grid hierarchy.
In addition to the problem of large kinks in the fine grid, we now have to deal with coarse grids having a large number of finite elements.
Here we want to simulate $7\cdot 7 \cdot 7$ inclusions resulting in a coarse grid consisting of $74\,096$ tetrahedral finite elements.
A coarser grid is not reasonable, since this would have a negative influence on the aspect ratios of the elements, which affects the convergence of iterative solvers.
Moreover, the large number of elements in the coarse grid is a major challenge for solvers based on multigrid.
In order to be efficient, a direct factorization is usually applied on the coarsest level, which is known to be only scalable to some extent.

Figure \ref{fig_compliance_parabolic_opt} shows some snapshots of the optimization on a fine grid with $4.74\cdot 10^6$ tetrahedral elements.
The aim of this particular model is to show how cellular structures can be fortified by tightly stacking cells together close to the surface, where forces are acting.
This, to some extent, reflects growing processes in biological skin structures.
In this particular case, we are not able to apply BFGS update techniques like in the pure parabolic test case.
The quasi-Newton method leads to step sizes which are too large to be feasible as deformations on the finite element grid.
We thus apply only a gradient descent method.
The geometry after 30 gradient steps, which is depicted in figure \ref{fig_final_grid}, is not yet optimal in the sense that it is a stationary point.
Due to increasing aspect ratios of discretization elements, the multigrid solver does not converge anymore.
This illustrates the challenge to choose proper regularization in the case of shape optimization, such that an optimal solution can still be represented in a finite element mesh.

\section{Conclusions}
One of the main focuses of this paper is to describe the interaction of algorithmic building blocks ranging from a multigrid finite element solver to quasi-Newton methods for shape optimization problems.
Here we concentrate on shape spaces and metrics, which are especially suited for large-scale computations.
This results in scalable algorithms for supercomputers. 
Based on complex examples this paper presents the possibilities and challenges of multigrid methods and shape optimization.
In particular, it is shown how to handle the increasing values of approximated curvature in an hierarchical grid structure affecting the regularization.
The purpose of the underlying models within this article is to form building blocks for further studies of biological cell structures.

\section*{Acknowledgment}
This work has been partly supported by the German Research Foundation (DFG) within the priority program SPP 1648 ``Software for Exascale Computing'' under contract number Schu804/12-1, the priority program SPP 1962 ``Non-smooth and Complementarity-based Distributed Parameter Systems: Simulation and Hierarchical Optimization'' under contract number Schu804/15-1 and the research training group 2126 ``Algorithmic Optimization''.

\bibliographystyle{plain}
\bibliography{citations.bib}

\end{document}